\documentclass[reqno,a4paper, 11pt]{amsart}

\usepackage[utf8]{inputenc}

\usepackage[a4paper=true,pdfpagelabels]{hyperref}
\usepackage{graphicx}

\usepackage{amsfonts,epsfig}
\usepackage{latexsym}
\usepackage{mathabx}
\usepackage{amsmath}
\usepackage{amssymb}

\usepackage{color}

\newtheorem{theorem}{Theorem}
\newtheorem{lemma}[theorem]{Lemma}

\newtheorem{proposition}[theorem]{Proposition}

\newtheorem{lettertheorem}{Theorem}

\theoremstyle{definition}

\theoremstyle{remark}

\numberwithin{equation}{section}

\newcommand{\D}{\mathbb{D}}
\newcommand{\DD}{\widehat{\mathcal{D}}}
\newcommand{\Dd}{\widecheck{\mathcal{D}}}
\newcommand{\M}{\mathcal{M}}
\newcommand{\DDD}{\mathcal{D}}

\newcommand{\N}{\mathbb{N}}

\newcommand{\C}{\mathbb{C}}

\renewcommand{\phi}{\varphi}

\def\a{\alpha}       \def\b{\beta}        
           
     \def\om{\omega}      
\def\s{\sigma}              
                  \def\z{\zeta}

\def\omg{\widehat{\omega}}
\def\nug{\widehat{\nu}}

\def\sg{\widehat{\sigma}}

\def\at{\widetilde{\alpha}}

\def\bt{\widetilde{\beta}}

\renewcommand{\H}{\mathcal{H}}

\newenvironment{Prf}{\noindent{\emph{Proof of}}}
{\hfill$\Box$ }

\begin{document}

\title[Bergman projection on Lebesgue space induced by doubling weight]{Bergman projection on Lebesgue space induced by doubling weight}


\author{Jos\'e \'Angel Pel\'aez}
\address{Departamento de An\'alisis Matem\'atico, Universidad de M\'alaga, Campus de
Teatinos, 29071 M\'alaga, Spain} \email{japelaez@uma.es}

\author{Elena de la Rosa}
\address{Departamento de An\'alisis Matem\'atico, Universidad de M\'alaga, Campus de
Teatinos, 29071 M\'alaga, Spain} 
\email{elena.rosa@uma.es}

\author{Jouni R\"atty\"a}
\address{University of Eastern Finland, P.O.Box 111, 80101 Joensuu, Finland}
\email{jouni.rattya@uef.fi}

\thanks{This research was supported in part by  La Junta de Andaluc{\'i}a,
project FQM210.}


\maketitle


\begin{abstract}

Let $\omega$ and $\nu$ be radial weights on the unit disc of the complex plane, and denote $\sigma=\omega^{p'}\nu^{-\frac{p'}p}$ and $\omega_x =\int_0^1 s^x \omega(s)\,ds$ for all $1\le x<\infty$. Consider the one-weight inequality
	\begin{equation}\label{ab1}
  \|P_\omega(f)\|_{L^p_\nu}\le C\|f\|_{L^p_\nu},\quad 1<p<\infty,\tag{\dag}
  \end{equation}
for the Bergman projection $P_\omega$ induced by $\om$. It is shown that the moment condition
	$$
	D_p(\omega,\nu)=\sup_{n\in\N\cup\{0\}}\frac{\left(\nu_{np+1}\right)^\frac1p\left(\sigma_{np'+1}\right)^\frac1{p'}}{\om_{2n+1}}<\infty
	$$
is necessary for \eqref{ab1} to hold. Further, $D_p(\omega,\nu)<\infty$ is also sufficient for \eqref{ab1} if $\nu$ admits the doubling properties $\sup_{0\le r<1}\frac{\int_r^1 \omega(s)s\,ds}{\int_{\frac{1+r}{2}}^1 \omega(s)s\,ds}<\infty$ and 
$\sup_{0\le r<1}\frac{\int_r^1 \omega(s)s\,ds}{\int_r^{1-\frac{1-r}{K}} \omega(s)s\,ds}<\infty$ for some $K>1$.
In addition, an analogous result for the one weight inequality
$
  \|P_\omega(f)\|_{D^p_{\nu,k}}\le C\|f\|_{L^p_\nu},
$
where 
	$$
	\Vert f \Vert_{D^p_{\nu, k}}^p
	=\sum\limits_{j=0}^{k-1}\vert f^{(j)}(0)\vert ^p+
	\int_{\D} \vert f^{(k)}(z)\vert^p (1-\vert z \vert )^{kp}\nu(z)\,dA(z)<\infty, \quad k\in \N,
	$$
is established. The inequality \eqref{ab1} is further studied by using the necessary condition $D_p(\om,\nu)<\infty$ in the case of the exponential type weights
$\nu(r)=\exp \left(-\frac{\alpha}{(1-r^l)^{\beta}} \right)$ and  $\om(r)= \exp \left(-\frac{\widetilde{\alpha}}{(1-r^{\widetilde{l}})^{\widetilde{\beta}}} \right)$,
where $0<\alpha, \, \widetilde{\alpha}, \,  l, \, \widetilde{l}<\infty$ and $0<\b , \, \widetilde{\beta}\le 1$. 	
  \end{abstract}

\maketitle

\section{Introduction and previous results}
Let $\H(\D)$ denote the space of analytic functions in the unit disc $\D=\{z\in\C:|z|<1\}$. A non-negative function $\om\in L^1=L^1(\D)$ such that $\om(z)=\om(|z|)$ for all $z\in\D$ is called a radial weight. For $0<p<\infty$ and such an $\omega$, the Lebesgue space $L^p_\om$ consists of complex-valued measurable functions $f$ on $\D$ such that
    $$
    \|f\|_{L^p_\omega}^p=\int_\D|f(z)|^p\omega(z)\,dA(z)<\infty,
    $$
where $dA(z)=\frac{dx\,dy}{\pi}$ is the normalized Lebesgue area measure on $\D$. The corresponding weighted Bergman space is $A^p_\om=L^p_\omega\cap\H(\D)$. Throughout this paper we assume $\widehat{\om}(z)=\int_{|z|}^1\om(s)\,ds>0$ for all $z\in\D$, for otherwise $A^p_\om=\H(\D)$.

For any radial weight $\om$, the norm convergence in $A^2_{\om}$ implies the uniform convergence on compact subsets of $\D$, and hence each point evaluation~$L_z$ is a bounded linear functional on~$A^2_\om$. Therefore there exist Bergman reproducing kernels $B^\om_z\in A^2_\om$ such that
	$$ 
	L_z(f)=f(z)=\langle f, B_z^{\om}\rangle_{A^2_\om}=\int_{\D}f(\z)\overline{B_z^{\om}(\z)}\om(\z)\,dA(\z),\quad f \in A^2_{\om}.
	$$
The Hilbert space $A^2_\om$ is a closed subspace of~$L^2_\om$, and hence the orthogonal projection from~$L^2_\om$ to~$A^2_\om$ is given by
    $$
    P_\om(f)(z) = \int_\D f(\zeta)\overline{B^\om_z(\zeta)}\om(\zeta)\,dA(\zeta),\quad z\in\D.
    $$
The operator $P_\om$ is the Bergman projection, and the maximal Bergman projection is defined by
	$$
	P_{\om}^+ (f)(z)=\int_{\D}f(\z)\vert B_{z}^{\om}(\z)\vert\om(\z)\,dA(\z), \quad z \in \D.
	$$
The kernel of the classical weighted Bergman space $A^2_\alpha$, induced by the standard radial weight $(\alpha+1)(1-|z|^2)^\alpha$ with $\alpha>-1$, is denoted by $B^\alpha_z$, and $P_\alpha$ stands for the corresponding Bergman projection.
The boundedness of projections on $L^p$-spaces is an intriguing topic which has attracted a considerable amount of attention during the last decades. This is not only due to the mathematical difficulties the question raises, but also to its numerous applications in important questions in operator theory such as duality relationships and Littlewood-Paley inequalities for weighted Bergman spaces, and the famous Sarason's conjecture on Bergman spaces \cite{AlPoRe,PR19,Zhu}.

Let $\nu$ be a radial weight and $1<p<\infty$. In this paper we are interested in describing the radial weights $\omega$ such that $P_\omega: L^p_\nu \to L^p_\nu$ is bounded. As far as we know, this is an open problem even in the case of the standard weight $\nu(z)=(\alpha+1)(1-|z|^2)^\alpha$. The classical Bergman kernel $B^\alpha_z$ has the very useful formula 
	$$
	B^\alpha_z(\zeta)=\frac1{(1-\overline{z}\zeta)^{2+\alpha}}, \quad z,\z\in \D.
	$$
However, if $\om$ is only assumed to be a radial weight, then the Bergman reproducing kernel $B^\omega_z$ associated to it does not have such a neat explicit expression, and this is precisely one of the main difficulties to tackle the aforementioned question. Consequently, we are forced to work with the identity
	\begin{equation}\label{eq:B}
	B^\om_z(\z)=\sum_{n=0}^\infty\frac{\left(\overline{z}\z\right)^n}{2\om_{2n+1}}, \quad z,\z\in \D,
	\end{equation}
where $\om_{x}=\int_0^1r^x\om(r)\,dr$ for all $x\ge0$. Therefore the influence of the weight to the kernel is transmitted by its odd moments through this infinite sum and that is pretty much all one can say in general about the kernel. 

Our first  result provides a useful necessary condition for $P_\omega: L^p_\nu \to L^p_\nu$ to be bounded. Here and on the following we write 
	$$
	\sigma=\sigma_{\om,\nu,p}=\left(\frac{\om}{\nu^\frac1p}\right)^{p'}=\frac{\om^\frac{p}{p-1}}{\nu^\frac1{p-1}}
	$$
for short.
 
\begin{proposition}\label{Proy. acotada cond. pesos JouniINTRO}
Let $\om$ and $\nu$ be radial weights and  $1<p<\infty$. If $P_\om:L^p_\nu\to L^p_\nu$ is bounded, then	
	$$
	D_p(\om,\nu)
	=\sup_{n\in\N\cup\{0\}}\frac{\left(\nu_{np+1}\right)^\frac1p\left(\sigma_{np'+1}\right)^\frac1{p'}}{\om_{2n+1}}
	\le\|P_\om\|_{L^{p}_\nu\to L^{p}_\nu}<\infty.
	$$
\end{proposition}

We do not know whether or not the condition $D_p(\om,\nu)<\infty$ guarantees the boundedness of $P_\om:L^p_\nu\to L^p_\nu$ for any pair $(\omega,\nu)$ of radial weights. In fact, almost two decades ago Dostani\'c posed the question of describing the radial weights $\om$ such that $P_\om$ is bounded on~$L^p_\om$, and showed that the condition 
	$$
	D_p(\om,\om)=\sup_{n\in\N\cup\{0\}}\frac{\left(\omega_{np+1}\right)^\frac1p\left(\omega_{np'+1}\right)^\frac1{p'}}{\om_{2n+1}}<\infty
	$$ 
is necessary for $P_\omega: L^p_\omega\to L^p_\omega$ to be bounded \cite{Dostanic}. Observe that $D_p(\om,\om)\ge1$ by H\"older's inequality, and hence the Dostani\'c condition $D_p(\om,\om)<\infty$ is just an asymptotic reverse H\"older condition.

In this study we will show that the condition $D_p(\om,\nu)<\infty$ implies the boundedness of $P_\om$ on $L^p_\nu$ for a large class of radial weights $\nu$ including all the standard weights. In order to give the precise statements some more definitions are in order. A radial weight $\om$ belongs to $\DD$ if there exists $C=C(\om)>0$ such that 
	$$
	\omg(r)\le C\omg\left(\frac{1+r}{2}\right),\quad r\to1^-,
	$$
and $\om\in\Dd$ if there exist $K=K(\om)>1$ and $C=C(\om)>1$ such that 
	$$
	\omg(r)\ge C\omg\left(1-\frac{1-r}{K}\right),\quad r\to1^-.
	$$ 
We denote $\mathcal{D}=\DD\cap\Dd$ for short, and we simply say that $\omega$ is a doubling weight if $\om\in\DDD$. The next few lines are dedicated to offer a brief insight to these classes of radial weights.
Doubling weights appear in a natural way in many questions on operator theory. For instance, the Bergman projection $P_\om$ acts as a bounded operator from the space $L^\infty$ of bounded complex-valued functions to the Bloch space $\mathcal{B}$ if and only if $\om\in\DD$, and the Littlewood-Paley formula
	\begin{equation}\label{LP}
	\|f\|_{A^p_\om}^p\asymp\int_\D|f'(z)|^p(1-|z|)^p\om(z)\,dA(z)+|f(0)|^p,\quad f\in A^p_\om,
	\end{equation}
is also equivalent to $\om\in\DDD$. Further, each standard radial weight obviously belongs to $\DDD$, while $\Dd\setminus\DDD$ contains exponential type weights such as
	$$
	\nu(r)=\exp \left(-\frac{\a}{(1-r^l)^{\b}} \right)\quad 0<\alpha,l,\beta<\infty.
	$$
The class of rapidly increasing weights, introduced in \cite{PelRat}, lies entirely within $\DD\setminus\DDD$, and a typical example of such a weight is 
	$$
	\nu(z)=\frac1{(1-|z|^2)\left(\log\frac{e}{1-|z|^2}\right)^{\alpha}},\quad 1<\alpha<\infty.
	$$
To this end we emphasize that the containment in $\DD$ or $\Dd$ does not require differentiability, continuity or strict positivity. In fact, weights in these classes may vanish on a relatively large part of each outer annulus $\{z:r\le|z|<1\}$ of $\D$. For basic properties of the aforementioned classes, concrete nontrivial examples and more, see \cite{P,PelRat,PR19} and the relevant references therein.
 
Apart from the moment condition $D_p(\om,\nu)<\infty$, we will naturally face the requirement of the finiteness of the quantity	
	\begin{equation*}
	A_p(\om,\nu)=\sup\limits_{0\le r<1}\frac{\left(\int_r^1\nu(t)t\,dt\right)^\frac1p\left(\int_r^1\sigma(t)t\,dt\right)^\frac1{p'}}{\int_r^1\om(t)t\,dt}.
	\end{equation*}
If $r$ is uniformly bounded away from one, then the quotient in the supremum above is certainly uniformly bounded, in particular, it equals to $\|\nu\|_{L^1}^\frac1p\|\sigma\|_{L^1}^\frac1{p'}/\|\om\|_{L^1}\in(0,\infty)$ for $r=0$ as all the weights involved are non-trivial radial weights. Therefore the boundedness of $A_p(\om,\nu)$ is equivalent to the fact that $\sigma\in L^1$ and
		$$
		\limsup_{r\to1^-}\frac{\nug(r)^{\frac{1}{p}}\widehat{\sigma}(r)^\frac1{p'}}{\omg(r)}
		=\limsup_{r\to1^-}\frac{\left(\int_r^1\nu(t)t\,dt\right)^\frac1p\left(\int_r^1\sigma(t)t\,dt\right)^\frac1{p'}}{\int_r^1\om(t)t\,dt}<\infty.
		$$
		
The next theorem is our first main result.

\begin{theorem}\label{th:main}
 Let $\om$ be a radial weight, $\nu \in \DDD$ and  $1<p<\infty$. Then the following statements are equivalent:
\begin{itemize}
\item[(i)] $P_{\om}:L^p_{\nu}\to L^p_{\nu}$ is bounded;
\item[(ii)] $P_{\om}^+:L^p_{\nu}\to L^p_{\nu}$ is bounded;
\item[(iii)] $D_p(\om,\nu)<\infty$;
\item[(iv)] $A_p(\om, \nu)<\infty$ and $\om \in \DD$.
\end{itemize}
\end{theorem}

%
%
%

It is known that the condition $A_p(\om,\nu)<\infty$ is closely related to the boundedness of~$P_{\om}$ and~$P_{\om}^+$ on $L^p_{\nu}$, under certain  hypotheses on the radial weights involved \cite{PR19}, but the condition $D_p(\om,\nu)<\infty$ seems to be new in this context. To be precise, the next result follows from \cite[Theorem~13]{PR19}. Here and from now on 
	\begin{equation*}\label{eq:Mpcondition}
  M_p(\om,\nu)
	=\sup_{0\le r<1}\left(\int_0^r\frac{\nu(s)s}{\left(\int_s^1\om(t)t\,dt\right)^p}\,ds+1\right)^\frac1p
	\left(\int_r^1\sigma(t)t\,dt\right)^{\frac{1}{p'}}.
  \end{equation*}
	
\begin{lettertheorem}\label{teorema 13PR19}
Let $1<p<\infty$ and $\om\in\DD$, and let $\nu$ be a radial weight. Then the following statements are equivalent:
		\begin{itemize}
    \item[(i)] $P_{\om}^+:L^p_{\nu}\to L^p_{\nu}$ is bounded;
    \item[(ii)] $A_p(\om,\nu)<\infty$ and $\omega,\nu\in\DDD$;
    \item[(iii)] $M_p(\om,\nu)<\infty$ and $\omega,\nu\in\DDD$.
		\end{itemize}
\end{lettertheorem}

To prove Theorem~\ref{th:main} we will show that Theorem~\ref{teorema 13PR19} can be improved in the sense that the boundedness of $P_{\om}^+:L^p_{\nu}\to L^p_{\nu}$ can be achieved under strictly weaker hypotheses on~$\omega$ and~$\nu$. We write $\om\in\M$ if there exist constants $C=C(\om)>1$ and $K=K(\om)>1$ such that $\om_{x}\ge C\om_{Kx}$ for all $x\ge1$. It is known that $\Dd\subsetneq\M$ and $\DDD=\DD\cap\M$ by \cite[Proof of Theorem~3 and Proposition~14]{PR19}.
 
\begin{theorem}\label{teorema 1}
Let $1<p<\infty$ and $\om\in\DD$, and let $\nu$ be a radial weight. Then the following statements are equivalent:
		\begin{itemize}
    \item[(i)] $P_{\om}^+:L^p_{\nu}\to L^p_{\nu}$ is bounded;
    \item[(ii)] $D_p(\om,\nu)<\infty$ and $\nu\in\M$;
    \item[(iii)] $A_p(\om,\nu)<\infty$ and $\nu\in\M$;
    \item[(iv)] $M_p(\om,\nu)<\infty$ and $\nu\in\M$.
		\end{itemize}
\end{theorem}

For a radial weight $\nu$, $k\in\N$ and $0<p<\infty$, the Dirichlet space $D^p_{\nu, k}$ consists of functions $f\in\H(\D)$ such that
	$$
	\Vert f \Vert_{D^p_{\nu, k}}^p
	=\sum\limits_{j=0}^{k-1}\vert f^{(j)}(0)\vert ^p+
	\int_{\D} \vert f^{(k)}(z)\vert^p (1-\vert z \vert )^{kp}\nu(z)\,dA(z)<\infty.
	$$
Obviously, $D^p_{\nu,1}=A^p_\nu$ by \eqref{LP}, and in general, the identity $D^p_{\nu,k}=A^p_\nu$ holds for $k\in\N$ if and only if $\nu\in\DDD$ \cite[Theorem~5]{PR19}. Moreover, a necessary condition (and also sufficient for even $p$'s) for the embedding $D^p_{\nu, k}\subset A^p_\nu$ is $\nu\in\mathcal{M}$ \cite[Section 8.1]{PR19}. Therefore the spaces~$D^p_{\nu, k}$ and~$A^p_\nu$ are different for $\nu\in \mathcal{M}\setminus \DD$. Our next result describes the boundedness of the Bergman projection $P_\omega: L^p_{\nu}\to D^p_{\nu, k}$ when $\om\in\DD$ and $\nu\in\M$. For $k\in\N$ and a radial weight $\omega$ we denote
	$$
	T_{\om, k}^+(f)(z)=(1-\vert z \vert)^k\int_{\D}f(\z)\left|(B_{\z}^{\om})^{(k)}(z)\right|\om (\z)\,dA(\z) , \quad z \in \D.
	$$
	
\begin{theorem}\label{teorema 2}
Let $1<p<\infty$, $k\in\N$, $\om\in\DD$ and $\nu\in\M$. Then the following statements are equivalent:
\begin{itemize}
    \item[(i)] $T_{\om, k}^+:L^p_{\nu}\to L^p_{\nu}$ is bounded;
    \item[(ii)] $P_{\om}:L^p_{\nu}\to D^p_{\nu, k}$ is bounded;
    \item[(iii)] $P_{\om}:L^p_{\nu}\to L^p_{\nu}$ is bounded;
    \item[(iv)] $D_p(\om, \nu)<\infty$;
    \item[(v)] $A_p(\om, \nu)<\infty$.
\end{itemize}
\end{theorem}

As for the proof of Theorem~\ref{teorema 2} let us mention that the proof of (ii)$\Rightarrow$(iii) is based in the fact that the inequality
	$$
	\| f\|_{D^p_{\nu, k}}\le C \|f \|_{A^p_\nu},\quad f \in \H(\D),
	$$
holds if and only if $\nu\in\DD$ \cite[Theorem~6]{PR19}. On the other hand, the implication  (v)$\Rightarrow$(i) uses  
precise asymptotic estimates of the integral means of the $k$th-derivative of the Bergman reproducing kernel $B^\omega_z$ obtained in \cite[Theorem~1]{PelRatproj}.

As we already mentioned, the boundedness of Bergman projections on Lebesgue spaces~$L^p_\nu$ is really useful to tackle natural questions on operator theory on weighted Bergman spaces $A^p_\nu$. The following interesting result shows an indirect and curious way to obtain this kind of information.

\begin{theorem}\label{th:curioso}
Let $1<p<\infty$, and let $\om$ and $\nu$ be radial weights. If $P_{\om}:L^p_{\nu}\to D^p_{\nu,k}$ is bounded, then $P_{\nu}:L^p_{\nu}\to D^p_{\nu, k}$ and $P_{\nu}:L^p_{\nu}\to L^p_{\nu}$ are bounded.
\end{theorem}

What Theorem~\ref{th:curioso} practically says is that if there exists a radial weight $\om$ such that $P_{\om}:L^p_{\nu}\to D^p_{\nu,k}$ is bounded, then $P_{\nu}$ must be bounded on $L^p_{\nu}$. In other words, in this setting it is the choice $\om=\nu$ that makes the inequality $\|P_\om(f)\|_{L^p_\nu}\lesssim\|f\|_{L^p_\nu}$ easiest to achieve.

Our last main result concerns exponential weights. The proof is based on a proper application of Proposition~\ref{Proy. acotada cond. pesos JouniINTRO}.

\begin{theorem}\label{th:exponenciales}
Let $1<p<\infty$, and let
	$$
	\nu (r)=\exp \left(-\frac{\a}{(1-r^l)^{\b}} \right)\quad\text{ and }\quad \om(r)= \exp \left(-\frac{\at}{(1-r^{\widetilde{l}})^{\bt}} \right), \quad 0\leq r <1,
	$$
where $0<\a,\at,l,\widetilde{l}<\infty$ and $0<\b,\bt\le1$. Then $P_\omega:L^p_\nu\to L^p_\nu$ is bounded if and only if $\beta=\bt$ and 
$\at= \frac{2\alpha}{p} \left( \frac{\widetilde{l}}{l}\right)^\beta$.
\end{theorem}

The statement in Theorem~\ref{th:exponenciales} is essentially known by \cite{BonetLuskyTaskinen,CP,HuLvSc,ZeyTams2012}. Our contribution consists of completing the picture for any $0<l,\widetilde{l}<\infty$ and showing the usefulness of the condition $D_p(\om,\nu)<\infty$ in the setting of exponential weights.

The rest of the paper is organized as follows. Section~\ref{s1} is devoted to proving Proposition~\ref{Proy. acotada cond. pesos JouniINTRO} 
and some results involving the conditions $D_p(\omega,\nu)<\infty$ and $A_p(\omega,\nu)<\infty$, and the classes of radial weights $\DD$ and $\M$. Theorems~\ref{th:main}, \ref{teorema 1}, \ref{teorema 2} and \ref{th:curioso} are proved in Section~\ref{s2}, while the proof of  Theorem~\ref{th:exponenciales} is given in Section~\ref{s3}.

To this end a couple of words about the notation used. The letter $C=C(\cdot)$ will denote an absolute constant whose value depends on the parameters indicated in the parenthesis, and may change from one occurrence to another.
We will use the notation $a\lesssim b$ if there exists a constant
$C=C(\cdot)>0$ such that $a\le Cb$, and $a\gtrsim b$ is understood
in an analogous manner. In particular, if $a\lesssim b$ and
$a\gtrsim b$, then we write $a\asymp b$ and say that $a$ and $b$ are comparable.

\section{Necessary conditions}\label{s1}

We begin this section by a simple lemma which will serve us while obtaining necessary conditions for the boundedness of $P_\om:L^p_\nu\to L^p_\nu$.

\begin{lemma}\label{Lemma:simple}
Let $\om$ be a radial weight. Then $\om\in\DD$ if and only if for some (equivalently for each) $1<q<\infty$ there exists a constant $C=C(\om,q)\ge1$ such that 
	\begin{equation}\label{eq:simple}
	\om_x\le C\om_{qx},\quad 1\le x<\infty.
	\end{equation}
\end{lemma}

\begin{proof}
If $\om\in\DD$ then there exist constants $C=C(\om)\ge1$ and $\eta=\eta(\om)>0$ such that $\om_x\le C\left(\frac{y}{x}\right)^\eta\om_y$ for all $1\le x\le y<\infty$ by \cite[Lemma~2.1(x)]{P}. By choosing $y=qx$ we obtain \eqref{eq:simple}, with $Cq^\eta$ in place of $C$, for each $1<q<\infty$. 

Conversely, assume that \eqref{eq:simple} is satisfied for some $1<q<\infty$. If $2\le q<\infty$, then $\om_x\le C\om_{qx}\le C\om_{2x}$ for all $1\le x<\infty$, and since $\widehat{\om}\left(1-\frac1x\right)\asymp\om_x$, as $x\to\infty$, by \cite[Lemma~2.1(vi)]{P}, we deduce $\om\in\DD$. If $1<q<2$, we fix $n\in\N$ such that $q^n\ge2$. Then $\om_x\le C\om_{qx}\le\cdots\le C^n\om_{q^nx}\le C^n\om_{2x}$ for all $1\le x<\infty$, and again it follows that $\om\in\DD$ by \cite[Lemma~2.1(vi)]{P}.
\end{proof}

The next result contains Proposition~\ref{Proy. acotada cond. pesos JouniINTRO}. Recall that 
	$$
	D_p(\om,\nu)
	=\sup_{n\in\N\cup\{0\}}\frac{\left(\nu_{np+1}\right)^\frac1p\left(\sigma_{np'+1}\right)^\frac1{p'}}{\om_{2n+1}},\quad 
	\sigma=\left(\frac{\om}{\nu^\frac1p}\right)^{p'}=\frac{\om^\frac{p}{p-1}}{\nu^\frac1{p-1}}.
	$$
	
\begin{proposition}
\label{Proy. acotada cond. pesos Jouni} Let $\om$ and $\nu$ be radial weights and $1<p<\infty$. 
		\begin{itemize}
		\item[\rm(i)]
		If $P_\om:L^p_\nu\to L^p_\nu$ is bounded, then $D_p(\om,\nu)\le\|P_\om\|_{L^{p}_\nu\to L^{p}_\nu}$.
		\item[\rm(ii)] If $ D_p(\om,\nu)<\infty$, then $\om\in\DD\,\Leftrightarrow\,\nu\in\DD$. Moreover, if $\om\in\DD$, then $\sigma\in\DD$.
		\item[\rm(iii)] If $ D_p(\om,\nu)<\infty$ and $\nu\in\M$, then $\omega\in\M$.
		\end{itemize}
\end{proposition}

\begin{proof}
(i). Let $f_n(\z)=\zeta^n \varphi(\z)$, where $\varphi$ is a radial function on $[0,1)$ and $n\in\N\cup\{0\}$. Then 
	$$
	P_\omega(f_n)(z)=\frac{\int_0^1\omega(s)\varphi(s)s^{2n+1}\,ds}{\omega_{2n+1}}z^n,\quad z\in\D.
	$$
Therefore
	\begin{equation}\label{eq:j1}
	\begin{split}
	\left(\frac{\int_0^1\varphi(s) \omega(s)\,s^{2n+1}\,ds}{\omega_{2n+1}}\right)^p \nu_{np+1} 
	&=\|P_\omega(f_n)\|^p_{L^p_\nu}
	\le\|P_\omega\|^p \int_0^1 \varphi^p(s)\nu(s)s^{np+1}\,ds,
	\end{split}
	\end{equation}
and hence $\omega(r)r$ is absolutely continuous with respect to $\nu(r)r$. Moreover, the proof of \cite[Lemma~B]{MOPR21} shows that $\sigma\in L^1$, and hence by choosing $\varphi(r)=\left(\frac{r^{(2-p)n}\omega(r)}{\nu(r)}\right)^{\frac{1}{p-1}}$ in \eqref{eq:j1}, we obtain
	$$
	D_p(\om,\nu)\le \|P_\om\|_{L^{p}_\nu\to L^{p}_\nu}.
	$$

(ii). Assume $D_p(\om,\nu)<\infty$, and let first $\om\in\DD$. Then $\om_{2n+1}\lesssim\om_{4n+1}$ for all $n\in\N\cup\{0\}$ by Lemma~\ref{Lemma:simple}. H\"older's inequality yields
	\begin{equation*}
	\left(\nu_{pn+1}\right)^\frac1p
	\le D_p(\om,\nu)\frac{\om_{2n+1}}{\left(\sigma_{p'n+1}\right)^\frac1{p'}}
	\lesssim D_p(\om,\nu)\frac{\om_{4n+1}}{\left(\sigma_{2p'n+1}\right)^\frac1{p'}}
	\le D_p(\om,\nu)\left(\nu_{2pn+1}\right)^\frac1p,
	\end{equation*}
and it follows that $\nu_{pn+1}\le D_p^p(\om,\nu)\nu_{2pn+1}$ for all $n\in\N\cup\{0\}$. Let $3\le x<\infty$, and choose $n\in\N\cup\{0\}$ such that $pn+1\le x<pn+2$. Then
	\begin{equation}\label{ljdshglasfdjh}
	\nu_x
	\le\nu_{pn+1}
	\le D_p^p(\om,\nu)\nu_{2pn+1}
	\le D_p^p(\om,\nu)\nu_{2x-4}\asymp \nu_{2x},\quad 3\le x<\infty,
	\end{equation}
and thus $\nu\in\DD$ by Lemma~\ref{Lemma:simple}. The proof of the fact that $\sigma\in\DD$, provided $\om\in\DD$, is analogous to the argument just presented, and hence it is omitted. 

Assume next $\nu\in\DD$. Take $y\in\mathbb{R}$ such that $1-p'<y<2-p'$. Then H\"older's inequality yields
	$$
	\om_{n(p'+y)+1}\le \left(\sigma_{np'+1}\right)^{\frac{1}{p'}} \left(\nu_{n(p'+py)+1}\right)^{\frac{1}{p}},
	$$
where $p'+y$ and $p'+py$ both are strictly positive by the choice of $y$. This together with the hypothesis $ D_p(\om,\nu)<\infty$ gives
	\begin{equation}\label{eq:j2}
	\sup_{n\in\N\cup\{0\}}
	\frac{\om_{n(p'+y)+1}}{\om_{2n+1}}\left(\frac{\nu_{np+1}}{\nu_{n(p'+py)+1}}\right)^{\frac{1}{p}}\le D_p(\om,\nu)<\infty.
	\end{equation}
Since $p'+y<2$ is equivalent to $p>p'+yp$, the quotient $\frac{\nu_{np+1}}{\nu_{n(p'+py)+1}}$ is bounded from below by the assumption $\nu\in\DD$ and Lemma~\ref{Lemma:simple}. Thus $\om_{n(p'+y)+1}\lesssim\om_{2n+1}$ for all $n\in\N\cup\{0\}$. Since $p'+y<2$ by the choice of $y$, we may argue as in \eqref{ljdshglasfdjh} to deduce $\om\in\DD$.

(iii). Assume $D_p(\om,\nu)<\infty$, and let $\nu\in\M$. It is easy to see that, by the assumption $D_p(\om,\nu)<\infty$, there exists a constant $C_1=C_1(\omega,\nu,p,p')>0$ such that
	$$
	\omega_x\ge C_1 \left(\nu_{\frac{p}{2}x}\right)^{\frac{1}{p}}\left(\sigma_{\frac{p'}{2}x}\right)^{\frac{1}{p'}},\quad x\ge 2.
	$$
Moreover, since $\nu\in\M$, by \cite[(2.17)]{PR19} there exist constants $\eta=\eta(\nu)>0$ and $C_2=C_2(\nu)>0$ such that 
	$$
	\nu_x\ge C_2 \left(\frac{y}{x}\right)^\eta \nu_y,\quad 1\le x\le y<\infty.
	$$
Therefore
	\begin{equation*}
	\begin{split}
	\omega_{x}&\ge C_1\left(\nu_{\frac{p}{2}x}\right)^{\frac{1}{p}}\left(\sigma_{\frac{p'}{2}x}\right)^{\frac{1}{p'}}
	\ge C_1 \left(C_2 K^\eta \right)^{\frac1p}\left(\nu_{\frac{p}{2}Kx}\right)^{\frac{1}{p}}\left(\sigma_{\frac{p'}{2}Kx}\right)^{\frac{1}{p'}}
	\ge C_1 \left(C_2 K^\eta \right)^{\frac1p} \omega_{Kx},\quad x\ge2.
	\end{split}
	\end{equation*}
Consequently, by choosing $K>1$ sufficiently large such that $C_1\left(C_2 K^\eta \right)^{\frac1p}>1$, we deduce $\omega\in\M$.
\end{proof}

Next, we prove some other results which  will be used in the proof of Theorem~\ref{th:main}. Recall that
	\begin{equation*}
  M_p(\om,\nu)
	=\sup_{0\le r<1}\left(\int_0^r\frac{\nu(s)s}{\left(\int_s^1\om(t)t\,dt\right)^p}\,ds+1\right)^\frac1p
	\left(\int_r^1\sigma(t)t\,dt\right)^{\frac{1}{p'}}.
  \end{equation*}
	
\begin{lemma}\label{lema Ap+ cond pesos} Let $1<p<\infty$ and $\om\in\DD$, and let $\nu$ be a radial weight.
\begin{itemize}
\item[\rm(i)] If 
 $A_p(\om,\nu)<\infty$, then $\nu,\,\s\in\DD$.
\item[\rm(ii)] $D_p(\om,\nu)<\infty$ if and only if $A_p(\om,\nu)<\infty$.
\item[\rm(iii)] If $A_p(\om,\nu)<\infty$, and either $\nu\in\M$ or $\sigma\in\M$, then $\om\in\Dd$.
\item[\rm(iv)] If $M_p(\om,\nu)<\infty$, then $\nu\in\DD$.
\end{itemize}
\end{lemma}

\begin{proof}
(i) Assume that $A_p(\om,\nu)<\infty$. Hölder's inequality yields 
	\begin{equation}\label{holder}
	\int_r^1\om(t)t\,dt\le\left(\int_r^1\nu(t)t\,dt\right)^\frac1p\left(\int_r^1\sigma(t)t\,dt\right)^\frac1{p'},\quad 0\le r<1.
	\end{equation}
The hypothesis $\om\in\DD$ combined with this inequality gives
	\begin{equation*}
	\begin{split}
  \left(\int_r^1\nu(t)t\,dt\right)^\frac1p
	&\le A_p(\om,\nu)\frac{\int_r^1\om(t)t\,dt}{\left(\int_r^1\sigma(t)t\,dt\right)^\frac1{p'}}
	\le A_p(\om,\nu)\frac{\int_r^1\om(t)t\,dt}{\left(\int_{\frac{1+r}{2}}^1\sigma(t)t\,dt\right)^\frac1{p'}}\\
	&\lesssim A_p(\om,\nu)\frac{\int_\frac{1+r}{2}^1\om(t)t\,dt}{\left(\int_{\frac{1+r}{2}}^1\sigma(t)t\,dt\right)^\frac1{p'}}
	\le A_p(\om,\nu)\left(\int_\frac{1+r}{2}^1\nu(t)t\,dt\right)^\frac1p,\quad 0\le r<1,
	\end{split}
	\end{equation*}
and hence $\nu\in\DD$. By interchanging the roles of $\nu$ and $\sigma$, the argument above also proofs the second statement $\s\in\DD$.

(ii) This is an immediate consequence of the case (i) just proved, Proposition~\ref{Proy. acotada cond. pesos Jouni}(ii) and \cite[Lemma~2.1]{P}.

(iii) We will prove only the case in which $\nu\in\M$ since the assertion under the hypothesis on $\sigma$ is obtained by the same argument. The part (i) just established implies $\nu\in\DD$, and since $\nu\in\M$ by the hypothesis, we deduce $\nu\in\DDD=\DD\cap\Dd$ by \cite[Theorem~3]{PR19}. By \cite[Lemma~B]{RosaPelaez2022} we know that $\nu\in\Dd$ if and only if there exist constants $C=C(\nu)>0$ and $\alpha=\alpha(\nu)>0$ such that
	$$
	\int_r^1\nu(s)s\,ds\ge C\left(\frac{1-r}{1-t}\right)^\alpha\int_t^1\nu(s)s\,ds,\quad 0\le r\le t<1.
	$$
Therefore
	\begin{equation*}
	\begin{split}
	\int_r^1\om(t)t\,dt
	&\ge\frac1{A_p(\om,\nu)}\left(\int_r^1\nu(t)t\,dt\right)^\frac1p\left(\int_r^1\sigma(t)t\,dt\right)^\frac1{p'}\\
	&\ge\frac{C^\frac1p K^\frac{\alpha}p}{A_p(\om,\nu)}\left(\int_{1-\frac{1-r}{K}}^1\nu(t)t\,dt\right)^\frac1p\left(\int_r^1\sigma(t)t\,dt\right)^\frac1{p'}\\
	&\ge\frac{C^\frac1p K^\frac{\alpha}p}{A_p(\om,\nu)}
	\left(\int_{1-\frac{1-r}{K}}^1\nu(t)t\,dt\right)^\frac1p\left(\int_{1-\frac{1-r}{K}}^1\sigma(t)t\,dt\right)^\frac1{p'}\\
	&\ge\frac{C^\frac1p K^\frac{\alpha}p}{A_p(\om,\nu)}\left(\int_{1-\frac{1-r}{K}}^1\om(t)t\,dt\right),\quad 0\le r<1,
	\end{split}
	\end{equation*}
where the last inequality follows from \eqref{holder}. By choosing $K=K(\om,\nu,p)>1$ sufficiently large such that $C^\frac1p K^\frac{\alpha}p>A_p(\om,\nu)$, we deduce $\om\in\Dd$.

(iv) By Hölder's inequality, we deduce
	\begin{equation*}
	\begin{split}
  \int_0^r\frac{\nu (s)s}{\int_s^1\nu(t)t\,dt\int_s^1\om(t)t\,dt}ds
	&\le\left(\int_0^r\frac{\nu(s)s}{\left(\int_s^1\om(t)t\,dt\right)^p}ds\right)^{\frac{1}{p}}
	\left(\int_0^r\frac{\nu(s)s}{\left(\int_s^1\nu(t)t\,dt\right)^{p'}}ds\right)^{\frac{1}{p'}}\\
	&\lesssim M_p(\om,\nu)\frac{\left(\int_r^1\nu(t)t\,dt\right)^{\frac{1}{p'}-1}}{\left(\int_r^1\sigma(t)t\,dt\right)^{\frac{1}{p'}}}
	\le\frac{M_p(\om,\nu)}{\int_r^1\om(t)t\,dt},\quad 0\le r<1.
  \end{split}
	\end{equation*}
By replacing $r$ by $\frac{1+r}{2}$, and using the hypothesis $\om\in\DD$, we deduce
	\begin{equation*}
	\begin{split}
  \frac{M_p(\om,\nu)}{\int_r^1\om(t)t\,dt}
	&\gtrsim\frac{M_p(\om,\nu)}{\int_{\frac{1+r}{2}}^1\om(t)t\,dt}
	\gtrsim\frac1{\int_r^1\om(t)t\,dt}\int_r^{\frac{1+r}{2}}\frac{\nu (s)s}{\int_s^1\nu(t)t\,dt}ds\\
	&=\frac1{\int_r^1\om(t)t\,dt}\log\frac{\int_r^1\nu(t)t\,dt}{\int_\frac{1+r}{2}^1\om(t)t\,dt},\quad 0\le r<1,
  \end{split}
	\end{equation*}
and it follows that $\nu\in\DD$.
\end{proof}

\section{Main results}\label{s2}

Now we are ready to prove Theorems~\ref{th:main} and~\ref{teorema 1}. We begin with the latter one.

\medskip

\begin{Prf}{\em Theorem~\ref{teorema 1}.}
Assume first (i), that is, $P_{\om}^+:L^p_{\nu}\to L^p_{\nu}$ is bounded. Then $\nu\in\M$ by \cite[Proposition~20]{PR19} and $D_p(\om,\nu)<\infty$ by Proposition~\ref{Proy. acotada cond. pesos Jouni}(i). Thus (ii) is satisfied. The fact that (ii) implies (iii) is an immediate consequence of Lemma~\ref{lema Ap+ cond pesos}(ii). Assume now (iii). Then $\nu\in\DD$ by Lemma~\ref{lema Ap+ cond pesos}(i) and $\om\in\Dd$ by Lemma~\ref{lema Ap+ cond pesos}(iii). Since $\DD\cap\M=\DD\cap\Dd$ by \cite[Theorem~3]{PR19}, we deduce $\om,\,\nu\in\DDD$, and hence  Theorem~\ref{teorema 13PR19} yields (i). Thus (i), (ii) and (iii) are equivalent.

For any radial weights $\om$ and $\nu$ the condition $A_p(\om,\nu)<\infty$ implies $M_p(\om,\nu)<\infty$ by \cite[pp. 55--56]{PR19}. Thus (iii) implies (iv). To complete the proof, assume (iv). Since $\DD\cap\M=\DD\cap\Dd$ by \cite[Theorem~3]{PR19}, Lemma~\ref{lema Ap+ cond pesos}(iv) yields $\nu\in\DD\cap\M=\DDD$. The proof of \cite[(6.3)]{PR19} shows that $A_p(\om,\nu)\lesssim M_p(\om,\nu)$ if $\om\in\DD$ and $\nu\in\Dd$. Thus (iii) holds. This finishes the proof of the theorem.
\end{Prf}

\medskip

\begin{Prf}{\em{Theorem~\ref{th:main}.}} Obviously, (ii) implies (i). Further, the implication (i)$\Rightarrow$(iii) follows by Proposition~\ref{Proy. acotada cond. pesos Jouni}(i). Assume now (iii). Since $\nu\in\DD$ by the hypothesis, we have $\omega\in\DD$ by Proposition~\ref{Proy. acotada cond. pesos Jouni}(ii). Therefore $A_p(\om, \nu)<\infty$ by Lemma~\ref{lema Ap+ cond pesos}(ii), and hence (iv) is satisfied. Finally, (iv)$\Rightarrow$(ii) by Theorem~\ref{teorema 1}, and thus the proof is complete.
\end{Prf}

Some more notation is needed before proving the next result which is essential to obtain Theorem~\ref{teorema 2}.
For $f\in \H(\D)$, write
		$$
    M_1(r,f)=\frac{1}{2\pi}\int_0^{2\pi}|f(re^{i\theta})|\,d\theta,\quad 0\le r<1.
    $$

\begin{proposition}
\label{Suficiencia acotacion Twk+}
Let $1<p<\infty$, $k\in\N$ and $\om\in\DD$, and let $\nu$ be a radial weight such that $A_p(\om,\nu)<\infty$. Then $T_{\om, k}^+:L^p_{\nu}\to L^p_{\nu}$ is bounded.
\end{proposition}

\begin{proof}
The hypothesis $A_p(\om, \nu)<\infty$ and Hölder's inequality imply $L^p_{\nu}\subset L^1_{\om}$. Let $f\in L^p_{\nu}$. Then the function $z\mapsto\frac{T_{\om,\,k}^+(|f|)(z)}{(1-\vert z \vert)^k}$ is subharmonic in $\D$, and thus its integral means are non-decreasing. Therefore
		\begin{align*}
    \Vert T_{\om, k}^+(f)\Vert_{L^p_{\nu}}^p
		&\le\Vert T_{\om, k}^+(|f|)\Vert_{L^p_{\nu}}^p 
		\lesssim \int_{\D\setminus D(0,\frac{1}{2})} \left( T_{\om, k}^+(|f|)(z)\right)^p\nu(z)\,dA(z)\\
		& \lesssim\int_{\D\setminus D(0,\frac{1}{2})} (1-\vert z \vert)^{kp} \left(\int_{\D\setminus D(0,\frac{1}{2})} 
		\vert f(\zeta)\vert \vert (B_{\zeta}^{\om})^{(k)}(z)\vert \om (\zeta)\,dA(\zeta)\right)^p\nu(z)\,dA(z)\\
		&\quad+\int_{\D\setminus D(0,\frac{1}{2})} (1-\vert z \vert)^{kp} 
		\left(\int_{D(0,\frac{1}{2})} \vert f(\zeta)\vert \vert (B_{\zeta}^{\om})^{(k)}(z)\vert \om (\zeta)\,dA(\zeta)\right)^p\nu(z)\,dA(z)\\
		&\lesssim\int_{\D\setminus D(0,\frac{1}{2})} (1-\vert z \vert)^{kp} 
		\left(\int_{\D\setminus D(0,\frac{1}{2})} \vert f(\zeta)\vert \vert (B_{\zeta}^{\om})^{(k)}(z)\vert \om(\zeta)\,dA(\zeta)\right)^p
    \nu(z)\,dA(z)\\
		&\quad+\left(\int_{D(0,\frac{1}{2})} \vert f(\zeta)\vert \om (\zeta)\,dA(\zeta)\right)^p=I(f)+II(f).
		\end{align*}		
To deal with $I(f)$, consider the function 
	$$
	h(\zeta)=\nu(\zeta)^{\frac{1}{p}}\left(\int_{|\zeta|}^1\sigma(s)s\,ds\right)^\frac1{pp'},\quad \zeta\in\D,
	$$
which is well defined by the hypothesis $A_p(\om,\nu)<\infty$. By Hölder's inequality, we obtain
		\begin{equation}
		\begin{split}
		\label{polla}
	 I(f)
		&\le\int_{\D\setminus D(0,\frac{1}{2})}(1-\vert z \vert)^{kp}
		\left(\int_{\D\setminus D(0,\frac{1}{2})} \vert f(\zeta)\vert^p h(\zeta)^p \vert (B_{\zeta}^{\om})^{(k)}(z)\vert  dA(\zeta)\right)\\
		&\quad\cdot\left(\int_{\D}\vert (B_{\zeta}^{\om})^{(k)}(z)\vert \left(\frac{\om (\zeta)}{h(\zeta)}\right)^{p'}dA(\zeta)\right)^{\frac{p}{p'}}\nu(z)\,dA(z).		
		\end{split}
		\end{equation}
   Observe that an integration yields
		\begin{equation}\label{integracion}
    \int_t^1\left(\frac{\om (s)}{h(s)}\right)^{p'}s\,ds
		=p'\left(\int_t^1\left(\frac{\om (s)}{\nu(s)}\right)^{p'} \nu(s)s\,ds\right)^{\frac{1}{p'}}
		=p'\left(\int_t^1\sigma(s)s\,ds\right)^\frac1{p'},\quad 0\le t<1.
		\end{equation}
This together with \cite[Theorem~1]{PelRatproj} and Fubini's Theorem yield 
	\begin{align*}
  \int_{\D}\vert (B_{\zeta}^{\om})^{(k)}(z)\vert\left(\frac{\om (\zeta)}{h(\zeta)}\right)^{p'}dA(\zeta)
	&\lesssim\int_0^1\left(\frac{\om (s)}{h(s)}\right)^{p'} M_1(s, (B_{z}^{\om})^{(k)})s\,ds\\
	&\lesssim\int_0^1\left(\frac{\om (s)}{h(s)}\right)^{p'} \left(\int_0^{s\vert z \vert}\frac{dt}{(1-t)^{k+1}\int_t^1 s\omega(s)\,ds}+1\right)s\,ds\\
	&=\int_0^{\vert z \vert} \frac{1}{(1-t)^{k+1}\int_t^1 s\omega(s)\,ds}\left( \int_{t/\vert z \vert}^1 \left(\frac{\om (s)}{h(s)}\right)^{p'}s\, ds\right) dt +\frac{p'}{2\pi}\|\sigma\|_{L^1}^\frac1{p'}\\
	&\lesssim\int_0^{\vert z \vert} \frac{\left(\int_t^1\sigma(s)s\,ds\right)^\frac1{p'}}{(1-t)^{k+1}\int_t^1 s\omega(s)\,ds} dt+\|\sigma\|_{L^1}^\frac1{p'}\\
	&\lesssim A_p(\om,\nu)\left(\int_0^{\vert z \vert}\frac{dt}{\left(\int_t^1\nu(s)s\,ds\right)^{\frac{1}{p}} (1-t)^{k+1}}+1\right)\\
	&\lesssim \frac{A_p(\om, \nu)}{\left(\int_{|z|}^1\nu(s)s\,ds\right)^{\frac{1}{p}} (1-\vert z \vert )^{k}},\quad z\in\D.
	\end{align*}
By combining this estimate with \eqref{polla}, and by applying Fubini's theorem and \cite[Theorem~1]{PelRatproj} again, we deduce
\begin{align*}
    I(f) &\lesssim A_p^{p-1}(\om,\nu)\int_{\D\setminus D\left(0,\frac{1}{2}\right)}
		\left(\int_{\D\setminus D\left(0,\frac{1}{2}\right)} \vert f(\zeta)\vert^p h(\zeta)^p \vert (B_{\zeta}^{\om})^{(k)}(z)\vert  dA(\zeta)\right)  (1-\vert z \vert)^{k} \frac{\nu(z)}{\nug(z)^{\frac{1}{p'}}} dA(z)
\\
&\le A_p^{p-1}(\om,\nu)\int_{\D\setminus D\left(0,\frac{1}{2}\right)} \vert f(\zeta)\vert^p h(\zeta)^p \left(\int_{\frac{1}{2}}^1(1-s)^{k} \frac{\nu(s)}{\nug(s)^{\frac{1}{p'}}} M_1\left(s, (B_{\zeta}^{\om})^{(k)}\right)\,ds \right)dA(\zeta)
\\
		&\lesssim A_p^{p-1}(\om,\nu)\int_{\D\setminus D\left(0,\frac{1}{2}\right)}\vert f(\zeta)\vert^p h(\zeta)^p
		\left(\int_{\frac{1}{2}}^1(1-s)^{k}\frac{\nu(s)}{\nug(s)^{\frac{1}{p'}}}
		\left(\int_0^{s\vert \zeta \vert} \frac{1}{\omg(t)(1-t)^{k+1}}\,dt+1\right)  ds \right)dA(\zeta)
\\
		&\lesssim A_p^{p-1}(\om,\nu)\int_{\D\setminus D\left(0,\frac{1}{2}\right)}\vert f(\zeta)\vert^p h(\zeta)^p 
		\left(\int_{\frac{1}{2}}^1\frac{\nu(s)}{\nug(s)^{\frac{1}{p'}}}\frac{(1-s)^{k}}{\omg(s\vert \zeta \vert)(1-s\vert \zeta \vert)^{k}}\,ds \right)dA(\zeta)
\\
		&\le A_p^{p-1}(\om,\nu)\int_{\D\setminus D\left(0,\frac{1}{2}\right)}\vert f(\zeta)\vert^p h(\zeta)^p 
		\left(\int_{\frac{1}{2}}^1\frac{\nu(s)}{\omg(s\vert \zeta \vert)\nug(s)^{\frac{1}{p'}}}\,ds \right)dA(\zeta).
\end{align*}

We now split the integral over $(\frac12,1)$ into two parts at $|\zeta|$. On one hand,
\begin{equation}
\begin{split}
    \label{I}
    I_{|\zeta|}
		&=h(\zeta)^p\left(\int_{|\zeta|}^1\frac{\nu(s)}{\omg(s\vert \zeta \vert)\nug(s)^{\frac{1}{p'}}}\,ds \right)
		\le\frac{h(\z)^p}{\omg(\z)} \int_{|\z|}^1\frac{\nu(s)}{\nug(s)^{\frac{1}{p'}}}  ds\\
    &\le p\frac{\sg(\z)^{\frac{1}{p'}}\nug(\zeta)^{\frac{1}{p}}}{\omg(\z)} \nu(\z)
		\lesssim A_p(\om, \nu) \nu(\z),\quad \zeta\in\D\setminus D\left(0,\frac12\right).
	\end{split}    
	\end{equation}
On the other hand, 
		\begin{equation*}
    \begin{split}
		I^{|\zeta|}
		&=h(\zeta)^p 
		\left(\int_{\frac12}^{|\zeta|}\frac{\nu(s)}{\omg(s\vert \zeta \vert)\nug(s)^{\frac{1}{p'}}}\,ds \right)
		\le h(\zeta)^p \int^{|\z|}_{\frac{1}{2}} \frac{1}{\omg(s)} \frac{\nu(s)}{\nug(s)^{\frac{1}{p'}}}\,ds\\
		&=h(\zeta)^p \int^{|\z|}_{\frac{1}{2}} \frac{\omg(s)^{p-1} \left(\int_0^s\frac{\nu(t)}{\omg(t)^p}t\,dt\right)^{\frac{1}{p'}}}{\nug(s)^{\frac{1}{p'}}} \frac{\nu(s)}{\omg(s)^{p} \left(\int_0^s\frac{\nu(t)}{\omg(t)^p}t\, dt\right)^{\frac{1}{p'}}}\,ds\\
		&\lesssim M_p^\frac{p}{p'}(\om,\nu)h(\zeta)^p \int^{|\z|}_{\frac{1}{2}} \frac{\omg(s)^{p-1} }{\nug(s)^{\frac{1}{p'}}\widehat{\sigma}(s)^\frac{p}{(p')^2}} \frac{\nu(s)}{\omg(s)^{p} \left(\int_0^s\frac{\nu(t)}{\omg(t)^p}t\, dt\right)^{\frac{1}{p'}}}\,ds\\
		&\le M_p^\frac{p}{p'}(\om,\nu)h(\zeta)^p\int^{|\z|}_{\frac{1}{2}}
		\frac{\nu(s)}{\omg(s)^{p} \left(\int_0^s\frac{\nu(t)}{\omg(t)^p}t\, dt\right)^{\frac{1}{p'}}}\,ds\\
		&\lesssim M_p^{p-1}(\om,\nu)\nu(\zeta)\left(\int_{|\zeta|}^1\sigma(t)t\,dt\right)^\frac1{p'}\left(\int_0^{|\zeta|}\frac{\nu(t)}{\omg(t)^p}t\,dt\right)^\frac1p\\
		&\lesssim M_p^{p}(\om,\nu)\nu(\zeta),\quad \zeta\in\D\setminus D\left(0,\frac12\right).
    \end{split}
		\end{equation*}
Since $A_p(\om,\nu)<\infty$ implies $M_p(\om,\nu)<\infty$ by \cite[pp. 55--56]{PR19}, we deduce $I(f)\lesssim \Vert f\Vert _{L^p_{\nu}}^p$. Finally, by Hölder's inequality,
		\begin{align*}
    II(f)
		\le\Vert f\Vert_{L^p_{\nu}}^p
		\left(\sigma(D(0,1/2))\right)^{\frac{p}{p'}}
		\lesssim \Vert f\Vert_{L^p_{\nu}}^p,\quad f\in L^p_\nu,
		\end{align*}
and thus $T_{\om, k}^+:L^p_{\nu}\to L^p_{\nu}$ is bounded.
\end{proof}

\begin{Prf}{\em{Theorem~\ref{teorema 2}}.}
It is clear by the definitions that (i) implies (ii). Assume now (ii). Then $P_{\om}(f)=f$ for all $f \in A^p_{\nu}$, and hence $\|f\|_{D^p_{\nu,k}}\lesssim\|f\|_{A^p_{\nu}}$ for all $f\in\H(\D)$. Therefore $\nu\in\DD$ by \cite[Theorem~6]{PR19}. Now that $\nu\in\M$ by the hypothesis, and $\DDD=\DD\cap\M$ by \cite[Theorem~6]{PR19}, we deduce $A^p_{\nu}=D^p_{\nu,k}$ by \cite[Theorem~5]{PR19}. Therefore (iii) is satisfied. Further, Proposition \ref{Proy. acotada cond. pesos Jouni}(i) shows (iii)$\Rightarrow$(iv), and Lemma~\ref{lema Ap+ cond pesos}(ii) gives (iv)$\Rightarrow$(v). Finally, (iv) implies (i) by Proposition~\ref{Suficiencia acotacion Twk+}. This finishes the proof.
\end{Prf}


%

\medskip

Proposition~\ref{Suficiencia acotacion Twk+} says that $T_{\om,k}^+:L^p_{\nu}\to L^p_{\nu} $ is bounded if 
$\om \in \DD$ and $A_p(\om, \nu)<\infty$. However,   $P_{\om}^+:L^p_{\nu}\to L^p_{\nu}$ is not necessarily bounded under these hypotheses on
the weights involved. Indeed, 
if  $\om =\nu \in \DD\setminus \M$, the condition $A_p(\om, \nu)<\infty$ holds  but $P_{\om}^+$ is not bounded on $L^p_{\om}$   by \cite[Theorem~3 and Theorem~20]{PR19}. In fact, by \cite[Theorem~20]{PR19}, a necessary condition for $P_{\om}^+:L^p_{\nu}\to L^p_{\nu} $ to be bounded is that both weights $\om$ and $\nu$ belong to $\mathcal{M}$.

We finish the section by proving Theorem~\ref{th:curioso}.

\medskip

\begin{Prf}{\em{Theorem~\ref{th:curioso}.}}
First, observe that $P_{\om}$ is a well-defined operator on $L^p_{\nu}$ by the hypothesis, and hence $P_{\om}(f)=f$ for all $f\in A^p_{\nu}$. Therefore $\Vert f \Vert_{D^p_{\nu, k}}=\Vert P_{\om}(f) \Vert_{D^p_{\nu, k}}\lesssim \Vert f \Vert_{A^p_{\nu}} $ for all $f\in A^p_{\nu}$, and hence $\nu \in \DD$ by \cite[Theorem~6]{PR19}. It follows that $P_{\nu}:L^p_{\nu}\to L^p_{\nu}$ and $P_{\nu}:L^p_{\nu}\to D^p_{\nu, k}$ are bounded  operators by \cite[Theorems~7 and~11]{PR19}.
\end{Prf}

\section{Proof of Theorem~\ref{th:exponenciales}}\label{s3}

Recall that 
	$$
	\nu (r)=\exp \left(-\frac{\a}{(1-r^l)^{\b}} \right)\quad\text{ and }\quad \om(r)= \exp \left(-\frac{\at}{(1-r^{\widetilde{l}})^{\bt}} \right), \quad 0\leq r <1,
	$$
where $0<\a,\at,l,\widetilde{l}<\infty$ and $0<\b,\bt\le1$. The proof is divided into several steps. We begin with showing that $\beta=\bt$ is a necessary condition for $P_{\om}:L^p_{\nu}\to L^p_{\nu}$ be bounded. 
\medskip

\noindent{\bf{Step 1.}} Let us assume that $\beta\neq \bt$. We will show that $D_p(\om, \nu)=\infty$, and therefore $P_{\om}:L^p_{\nu}\to L^p_{\nu}$ is not bounded by Proposition \ref{Proy. acotada cond. pesos JouniINTRO}.

First, observe that 
	$$
	\s (r)=\exp \left( -\frac{p' \at }{(1-r^{\widetilde{l}})^{\bt}}+ \frac{p'}{p}\frac{\a}{(1-r^l)^{\b}} \right),\quad 0\le r<1,
	$$
and therefore $\s$ is not a weight if $\b>\bt$. Thus $D_p(\om,\nu)=\infty$ in the case $\b>\bt$.
 
Let now $\bt>\b$, and note that 
	$$
	\s (r)= \exp \left( -\frac{p' \at }{(1-r^{\widetilde{l}})^{\bt}} \left(1-\frac{\a}{\at p}\frac{(1-r^{\widetilde{l}})^{\bt}}{(1-r^l)^{\b}}\right)\right)\geq \exp \left( -\frac{p' \at }{(1-r^{\widetilde{l}})^{\bt}}\right),\quad 0\le r<1.
	$$
Then \cite[Lemma~2.1]{BonetLuskyTaskinen}, see also \cite[Lemma~1]{DostanicIbero},  yields
	\begin{equation*}
	\begin{split}
\om_{2n}&\asymp n^{-\frac{2+\bt}{2(\bt+1)}} \exp\left(-B\left(\at, \bt, \widetilde{l}\right)2^{\frac{\bt}{\bt+1}}n^{\frac{\bt}{\bt+1}}\right)\\
(\nu_{pn})^{\frac{1}{p}}&\asymp n^{-\frac{1}{p}\frac{2+\b}{2(\b+1)}} \exp\left(-B\left(\a, \b, l\right)p^{-\frac{1}{\b+1}}n^{\frac{\b}{\b+1}}\right),\\
(\s_{p'n})^{\frac{1}{p'}}&\gtrsim n^{-\frac{1}{p'}\frac{2+\bt}{2(\bt+1)}} \exp\left(-B\left(p'\at, \bt, \widetilde{l}\right)(p')^{-\frac{1}{\bt+1}}n^{\frac{\bt}{\bt+1}}\right)
	\end{split}\quad n\in\N, 
	\end{equation*}
where $B(\alpha,\beta,l)=l^{-\frac{\beta}{\beta+1}}\alpha^\frac1{\beta+1}\left(\beta^{\frac1{\beta+1}}+\beta^{-\frac1{\beta+1}}\right)$. Since
	$$
	-B(p'\at, \bt, \widetilde{l})(p')^{-\frac{1}{\bt+1}}+ B(\at, \bt, \widetilde{l})2^{\frac{\bt}{\bt+1}}
	=(\widetilde{l})^{-\frac{\bt}{\bt+1}} (\at)^{\frac{1}{\bt+1}} (\bt^{\frac{1}{\bt+1}}+ (\bt)^{-\frac{1}{\bt+1}} )(2^{\frac{\bt}{\bt+1}}-1)>0 
	$$
and $\frac{\bt}{\bt+1}> \frac{\b}{\b+1}$, we deduce
	$$
	\lim_{n\to \infty}\exp \left(n^{\frac{\bt}{\bt+1}}\left(-B(p'\at, \bt, \widetilde{l})(p')^{-\frac{1}{\bt+1}}+ B(\at, \bt, \widetilde{l})2^{\frac{\bt}{\bt+1}} \right)-n^{\frac{\b}{\b+1}}p^{-\frac{1}{\b+1}} B(\a, \b, l)\right)=\infty.
	$$
Consequently, $D_p(\om, \nu)=\infty$ and the first step of the proof is finished.

\medskip

As the second step we show that we can reduce the consideration to the case in which $l$ and $\widetilde{l}$ are iqual.

\medskip

{\bf{Step 2.}} Let $\beta=\bt$. Since $\lim_{r\to 1^-}\frac{1-r^{\widetilde{l}}}{1-r^l}=\frac{\widetilde{l}}{l}$ and $0<\beta\le 1$, a direct calculation shows that 
	$$
	\nu(r)\asymp \widetilde{\nu}(r)=\exp \left(-\left(\frac{\widetilde{l}}{l}\right)^\beta\frac{\a}{(1-r^{\widetilde{l}})^{\b}} \right),\quad 0\le r<1,
	$$
and therefore $L^p_\nu=L^p_{\widetilde{\nu}}$. It follows that it suffices to prove the case $\widetilde{l}=l$.

\medskip

We next show that if $\beta= \bt$ and $l=\widetilde{l}$, then $\at=\frac{2 \a}{p}$ is a necessary condition for $P_{\om}:L^p_{\nu}\to L^p_{\nu}$ to be bounded.

\medskip

{\bf{Step 3.}} Let $\beta=\bt$, $l=\widetilde{l}$ and $\at\neq\frac{2\a}{p}$. If $\at<\frac{\a}{p}$, then 
	$$
	\s(r)= \exp \left(-\frac{p'(\at-\frac{\a}{p})}{(1-r^l)^{\b}} \right),\quad 0\le r<1,
	$$
is not a weight, and hence $D_p(\om,\nu)=\infty$. Therefore $P_{\om}:L^p_{\nu}\to L^p_{\nu}$ is not bounded in this case by Proposition~\ref{Proy. acotada cond. pesos JouniINTRO}. 

If $\at=\frac{\a}{p}$, then \cite[Lemma~2.1]{BonetLuskyTaskinen} implies
	\begin{equation*}\label{moment estimations}
	\begin{split}
	\om_{2n}&\asymp 2^{-\frac{2+\b}{2(\b+1)}} n^{-\frac{2+\b}{2(\b+1)}} \exp\left(-B\left(\frac{\a}{p}, \b, l\right)2^{\frac{\b}{\b+1}}n^{\frac{\b}{\b+1}}\right)\\
	(\nu_{pn})^{\frac{1}{p}}&\asymp (pn)^{-\frac{1}{p}\frac{2+\b}{2(\b+1)}} \exp\left(-B\left(\a, \b, l\right)p^{-\frac{1}{\b+1}}n^{\frac{\b}{\b+1}}\right),\quad n\in\N,\\ 
	(\s_{np'})^{\frac{1}{p'}} &\asymp n^{-\frac{1}{p'}}.
	\end{split}
	\end{equation*}
Since 
	$$
	-B(\a, \b, l)p^{-\frac{1}{\b+1}}+B\left(\frac{\a}{p}, \b, l\right)2^{\frac{\b}{\b+1}}
	=l^{-\frac{\b}{\b+1}} \a^{\frac{1}{\b+1}} (\b^{\frac{1}{\b+1}}+\b^{-\frac{1}{\b+1}})p^{-\frac{1}{\b+1}}(2^{\frac{\b}{\b+1}}-1)>0,
	$$ 
we have $D_p(\om, \nu)=\infty$, and therefore $P_{\om}:L^p_{\nu}\to L^p_{\nu}$ is not bounded by Proposition~\ref{Proy. acotada cond. pesos JouniINTRO}.

Let now $\at > \frac{\a}{p}$, and observe that 
	$$
	B(\a, \b, l)p^{-\frac{1}{\b+1}}+ B\left(p'\left(\at-\frac{\a}{p}\right), \b, l \right)(p')^{-\frac{1}{\b+1}}-B\left(\at, \b, l\right)2^{\frac{\b}{\b+1}}<0
	$$
if and only if 
\begin{equation}
\label{i}
p^{-\frac{1}{\b +1}}\a^{\frac{1}{\b+1}} +\left(\at-\frac{\a}{p}\right)^{\frac{1}{\b+1}}-2^{\frac{\b}{\b+1}} \at^{\frac{1}{\b+1}}<0.
\end{equation}
Write $t=\frac{1}{\b+1}\in \left[\frac{1}{2},1\right)$, and consider the function $f:(\frac{\a}{p},\infty)\to \mathbb{R}$, defined by
	$$
	f(\at)=p^{-t}\a^{t} +\left(\at-\frac{\a}{p}\right)^{t}-2^{1-t} \at^{t}.
	$$ 
Then $f\left(\frac{2\a}{p}\right)=0$ and $f(\at)<0$ if $\at\neq \frac{2\a}{p}$. 
It follows that  $D_p(\om, \nu)=\infty$, and hence $P_{\om}:L^p_{\nu}\to L^p_{\nu}$ is not bounded by Proposition~\ref{Proy. acotada cond. pesos JouniINTRO}.

\medskip

We complete the proof by referring to the recent literature in order to see that $P_{\om}:L^p_{\nu}\to L^p_{\nu}$ is bounded whenever $\beta= \bt$,  $l=\widetilde{l}$ and $\at = \frac{2 \a}{p}$. 

\medskip

{\bf{Step 4.}} Let $\beta= \bt$,  $l=\widetilde{l}$ and $\at = \frac{2 \a}{p}$. Then $P_{\om}:L^p_{\nu}\to L^p_{\nu}$ is bounded by
\cite[Theorem ~1.3]{BonetLuskyTaskinen}, see also \cite[Theorem~4.1]{HuLvSc}.

\end{document}